\documentclass[12pt, reqno]{amsart}

\usepackage{amsmath,amsfonts,amsthm,amssymb,amsxtra}
\usepackage{bbm} 
\usepackage{hyperref}	

\usepackage[parfill]{parskip}
\usepackage{tikz, caption}
\usetikzlibrary{shapes}
\allowdisplaybreaks

\newtheorem{theorem}{Theorem}
\newtheorem*{theorem*}{Theorem}

\newtheorem{lemma}[theorem]{Lemma}

\newtheorem*{conjecture*}{Conjecture}

\theoremstyle{remark}
\newtheorem{remark}[theorem]{Remark}
\newtheorem{remarks}[theorem]{Remarks}

\theoremstyle{definition}

\numberwithin{equation}{section}
\newcommand{\Z}{{\mathbb Z}}
\newcommand{\R}{{\mathbb R}}
\newcommand{\C}{{\mathbb C}}

\newcommand{\N}{{\mathbb N}}
\newcommand{\T}{{\mathbb T}}

\renewcommand{\epsilon}{\varepsilon}

\newenvironment{proof*}
  {\par\pushQED{\qed}\normalfont\trivlist\item\relax}
  {\popQED\endtrivlist}

\newcommand{\Dt}{\,\mathcal{D}}

\begin{document}

\title[]{Sharp Asymptotics for Higher-Order Hardy Constants on Lattices}
\markright{}
\author[]{By Shubham Gupta}
\address{Department of Mathematics, Technion – Israel Institute of Technology, Haifa 32000, Israel.}
\email{sg1019@campus.technion.ac.il}

\thanks{This research was supported by ISF-DFG Lead Agency research grant (ISF’s No. 1531/25)}

\begin{abstract}
We study the optimal constants in higher-order Hardy inequalities on the lattice $\Z^d$. For each fixed $\ell \in \N$, we prove that the optimal constant $\mathcal{C}_\text{opt}^\ell(d)$ in
$$
\sum_{n \in \Z^d} |\Delta^{\ell/2}u(n)|^2 \geq \mathcal{C}_\text{opt}^\ell(d)\sum_{n \in \Z^d} \frac{|u(n)|^2}{|n|^{2\ell}}.
$$
satisfies
$$
\lim_{d\rightarrow\infty}\frac{\mathcal{C}_\text{opt}^\ell(d)}{d^\ell}
=2^\ell.
$$
The proof is based on a Fourier reduction to a family of singular Hardy inequalities on the flat torus, involving the weight
\[
\omega(x)^{-2\ell},
\qquad
\omega(x)^2=\sum_{j=1}^d\sin^2\left(\frac{x_j}{2}\right),
\]
and zero average condition on admissible functions. We establish these torus inequalities by combining a ground state representation formula with a weighted integrated Bochner identity in an iterative scheme. The method yields explicit constants, defined recursively in the order $\ell$, and requires only the classical unweighted Poincar\'e inequality on the torus. The appearance of the limiting constant $2^\ell$ is particularly striking, as it suggests that, in the high dimensional regime, the optimizers are localized near the unit sphere $\{n\in\mathbb Z^d:|n|=1\}$ in $\Z^d$.

\end{abstract}

\maketitle

\section{Introduction}\label{sec: intro} 
The classical Hardy inequality on $\R^d$ reads as: for $d \geq 3$ and $u \in C_0^\infty(\R^d)$ we have  
\begin{equation}\label{classical Hardy}
    \int_{\R^d} |\nabla u(x)|^2 dx \geq \frac{(d-2)^2}{4}\int_{\R^d} \frac{|u(x)|^2}{|x|^2} dx.
\end{equation}
This and other forms of Hardy inequalities have been extensively studied in the continuum as they provide fundamental tools across PDEs, spectral theory, mathematical physics, and probability \cite{balinsky, davies}. One of the prominent features of this inequality is that the inverse square Hardy weight enjoys many desirable optimality properties. Particularly it is \emph{critical}, meaning that the weight cannot be increased pointwise. As a consequence, the constant in \eqref{classical Hardy} is sharp, that is, it cannot be replaced by a strictly bigger real number.    

In this article, we shall be concerned with the discrete analogue of classical Hardy inequality \eqref{classical Hardy} on $\Z^d$, and our focus would be on the optimality of the constant. In dimension one, it is an old well-known result that goes back to Hardy \cite{kufner}, and the optimal constant is known to be $1/4$ (coinciding with the corresponding constant on $\R$). The higher dimensional case $d \geq 2$ was studied much later, first by Rozenblum and Solomyak \cite{RS1, RS2}, and then by Kapitanski and Laptev \cite{laptev-kapitanski}. However, the question of optimality of constant was left open in these works. Afterwards, inspired by the works in the continuum \cite{pinchover}, Keller, Pogorzelski and Pinchover \cite{kpp} used the \emph{supersolution construction} method to construct \emph{optimal Hardy weights} on infinite graphs. This method allows to construct optimal weights from positive superharmonic functions. Here optimality of a Hardy weight is a strong notion, which comprises three conditions: criticality, non-existence of minimizers, and optimality near infinity (in fact the first two conditions implies the third one, see for instance \cite{hake}). The last one is the most relevant for us and it means that outside of finite sets, one cannot increase the weight by a factor bigger than one. Their methods, when specialized on $\Z^d$, for $d \geq 3$, produce the following optimal Hardy weight as $|n| \rightarrow \infty$
\begin{equation}\label{Green weight asymptotics}
    w_\text{opt}(n) = \frac{\Delta \sqrt{G}}{\sqrt{G}} = \frac{(d-2)^2}{4} \frac{1}{|n|^2} + \mathcal{O}\left(\frac{1}{|n|^3}\right), 
\end{equation}
where $G$ denotes the minimal positive Green function of $\Delta$ on $\Z^d$. Notably, the leading term coincides with the classical weight as well as the continuum constant and has some lower order correction terms. From the asymptotics \eqref{Green weight asymptotics} and the optimality near infinity of $w_\text{opt}$, it follows that the best constant $\mathcal{C}_\text{opt}^1(d)$ in the classical Hardy inequality (the superscript $1$ stands for the order of the gradient $\nabla$, see \eqref{def: Hardy constant}) is at most $(d-2)^2/4$. One naturally wonders if the upper bound is attained, in particular if the correction terms in \eqref{Green weight asymptotics} are non-negative, which turns out to be a hard problem. The main difficulty in determining the sign of the lower order terms is the non-explicit nature of the Green function on $\Z^d$, making it challenging to use supersolution approach to determine $\mathcal{C}_\text{opt}^1(d)$. 

In \cite{gupta}, the author observed an interesting phenomenon that contrasts with the corresponding situation in the continuum. It was shown that the best constant $\mathcal{C}_\text{opt}^1(d)$ satisfies
$$
0 < \liminf_{d \rightarrow \infty} \frac{\mathcal{C}_\text{opt}^1(d)}{d} \leq \limsup_{d \rightarrow \infty}\frac{\mathcal{C}_\text{opt}^1(d)}{d} < \infty.
$$
This shows that the optimal Hardy constant in $\Z^d$ grows linearly in $d$ for large dimensions, in contrast to the quadratic growth of the continuum constant in $\R^d$ \eqref{classical Hardy}. Interestingly, this fact yields information about the asymptotic behaviour of the Green function. This implicitly shows that, in sufficiently large dimensions, the correction terms in the optimal weight \eqref{Green weight asymptotics} cannot all be non-negative. Consequently, the classical Hardy inequality cannot be deduced from \eqref{Green weight asymptotics} in the high-dimensional regime.

This work, however, left a natural question unanswered: is the linear growth sharp? That is, does the limit
$$\lim_{d \rightarrow \infty} \frac{\mathcal{C}_\text{opt}^1(d)}{d}$$
exist? The methods of this paper allow us to show that it does and the limit can be computed explicitly. The arguments can be scaled to yield similar statements for operators of arbitrary order $\Delta^{\ell/2}$, for $\ell \in \N$. In passing, we mention several related works concerning Hardy inequalities for higher order operators \cite{kpp-rellich, keller-lemm, gupta2, stampach1, huang-Ye1D,  stampach3}.

\subsection{Main results}
Let $\nabla$ and $\Delta$ be the discrete gradient and Laplacian on $\Z^d$ given by
$$
\nabla u(n) = (\nabla_1 u(n), \dots, \nabla_d u(n)), \hspace{11pt}\Delta u(n) = \sum_{j=1}^d \Big(2 u(n)- u(n-e_j) - u(n+ e_j) \Big), 
$$
where $\nabla_j$ is the difference operator in the $j^{th}$ direction 
$$
\nabla_j u(n) = u(n)-u(n-e_j),
$$
and $e_j$ is the $j^{th}$ standard basis of $\R^d$. 

Let $\ell \in \N$, then the powers of $\Delta$ are given by
\begin{equation}\label{def: powers of Laplacian}
    \begin{split}
    \Delta^{\ell/2} u := 
    \begin{cases}
        \Delta^{k} u,  &\text{if} \hspace{5pt}\ell = 2k, \vspace{5pt}\\
        \nabla (\Delta^k u),  &\text{if} \hspace{5pt} \ell = 2k+1.
    \end{cases}
    \end{split}
\end{equation}
Let $u \in C_c(\Z^d)$ be a finitely supported function on $\Z^d$ with $u(0) = 0$. Let $\mathcal{C}_\text{opt}^\ell(d)$ denotes the optimal constant in the discrete Hardy inequality for the operator $\Delta^{\ell/2}$ 
\begin{equation}\label{def: Hardy constant}
    \sum_{n \in \Z^d} |\Delta^{\ell/2}u(n)|^2 \geq \mathcal{C}_\text{opt}^\ell(d)\sum_{n \in \Z^d} \frac{|u(n)|^2}{|n|^{2\ell}}.
\end{equation}
In \cite{gupta}, the author proved that the best constant behaves asymptotically as
$$
0 < \liminf_{d \rightarrow \infty} \frac{\mathcal{C}_\text{opt}^\ell(d)}{d^\ell} \leq \limsup_{d \rightarrow \infty}\frac{\mathcal{C}_\text{opt}^\ell(d)}{d^\ell} < \infty.
$$
For comparison, the corresponding constant in $\R^d$ grows as $d^{2\ell}$ as $d \rightarrow \infty$ \cite{davies-hinz, yafaev}. Our main result establishes the existence of the above limit and computes its value explicitly.

\begin{theorem}\label{thm: discrete asymptotics}
Let $\mathcal{C}_\text{opt}^\ell(d)$ denote the best constant in \eqref{def: Hardy constant}. Then we have
$$
\lim_{d \rightarrow \infty} \frac{\mathcal{C}_\text{opt}^\ell(d)}{d^\ell} = 2^{\ell}.
$$
\end{theorem}

\begin{remarks}
\begin{itemize}
    \item [(1)] The proof of the asymptotics is quantitative. We derive explicit upper and lower bounds on $\mathcal{C}_\text{opt}^\ell(d)$ for $d > 2\ell$, see \eqref{def: inductive constants (lattice)} and Theorem \ref{thm: explict lattice constant}.
    \item[(2)] The limiting value is particularly interesting. Taking $u$ to be the characteristic function of the unit sphere $\{|n|=1\} \subseteq \Z^d$ in \eqref{def: Hardy constant} yields an upper bound $2^\ell d^{\ell} + \mathcal{O}(d^{\ell-1})$, refer to \eqref{limsup of lattice constant}. This suggests that for $d$ large, the optimizers of the Hardy inequality \eqref{def: Hardy constant} are localized around the unit sphere in $\Z^d$.  
    \item[(3)] Huang and Ye independently proved the same result in \cite{huang-Ye2}. Although the final discrete statements coincide, their approach differs substantially from ours. In particular, the two proofs yield several distinct intermediate results, many of which are of independent interest and complement one another. We refer the reader to Remark \ref{rem: comparison with Huang-Ye} for a more detailed comparison.    
    \end{itemize}
\end{remarks}

The proof the Theorem \ref{thm: discrete asymptotics} is based on ideas developed in \cite{laptev-kapitanski, gupta}. A key ingredient is the conversion of Hardy inequalities \eqref{def: Hardy constant} on $\Z^d$ to an equivalent Hardy-type integral inequalities on the flat torus, thereby bypassing the difficulties arising from the discrete nature of the problem. To this end, we borrow some notations and a starting result from \cite{gupta}.

Let $Q_d := (-\pi, \pi)^d$ denote the open cube in $\R^d$ and let $\psi: \overline{Q_d} \rightarrow \C$ be a function on its closure. We say $\psi$ is $2\pi$-periodic in each variable if 
$$
\psi(x_1, \dots, x_i, -\pi, x_{i+1}, \dots, x_d) = \psi(x_1, \dots, x_i, \pi, x_{i+1}, \dots, x_d),
$$
for all $1 \leq i \leq d$. We denote the corresponding space of smooth functions by
$$
C_\text{per}^\infty (\overline{Q_d}) := \Big\{\psi \in C^\infty(\overline{Q_d}): \psi \hspace{3pt} \text{and its derivatives are} \hspace{3pt}2\pi\text{-periodic in each variable}\Big\}.
$$
Equivalently, it describes the space of smooth functions on the flat torus $\T^d := [-\pi, \pi]^d$ with opposite faces identified. The following key result from \cite[Lemmas 2.2 and 2.3]{gupta} converts \eqref{def: Hardy constant} to a similar question on the torus $\T^d$. On the torus, we use the usual continuum gradient and Laplacian, $\nabla = (\partial_{x_1}, \dots, \partial_{x_d})$, $\Delta=\sum_{i=1}^d \partial_{x_i} \partial_{x_i}$, and its powers are defined as \eqref{def: powers of Laplacian}.

\begin{lemma}\label{lem: discrete to continuous}
Let $\ell \in \N $. Let $ u\in C_c(\Z^d)$ with $u(0)=0$. There exist a function $\psi \in C_\text{per}^\infty(\overline{Q_d})$ having zero average
$$
\int_{Q_d} \psi = 0,
$$
such that
\begin{align*}
    \sum_{n \in \Z^d} \frac{|u(n)|^2}{|n|^{2 \ell}} = \int_{Q_d} |\Delta^{\ell/2} \psi(x)|^2 dx, \hspace{5pt}\text{and} \hspace{5pt} \sum_{n \in \Z^d} |\Delta^{\ell/2} u(n)|^2 = 4^\ell \int_{Q_d} |\Delta^\ell \psi(x)|^2 \omega(x)^{2\ell} dx,
\end{align*}
where 
\begin{align*}
    \omega(x)^2 := \sum_{j=1}^d \sin^2(x_j/2).
\end{align*}
\end{lemma}

\begin{remark}
The existence result is semi-explicit, and the corresponding function $\psi$ can be expressed in terms of the Fourier transform of $u$ in an \emph{almost} explicit way. Since Fourier transforms are involved, the discrete condition $u(0) = 0$ is transferred to zero average condition on $\psi$.
\end{remark}
As a consequence of the preceding lemma, the discrete Hardy inequality \eqref{def: Hardy constant} can be deduced from corresponding estimates on the torus. We therefore study the associated Hardy-type inequalities on the torus and determine the asymptotic behaviour of their sharp constants.

\begin{theorem}\label{thm: torus asymptotics}
Let $\ell \in \N$ and $ \psi \in C_\text{per}^\infty(\overline{Q_d)}$ having zero average. Let $\mathcal{K}_\text{opt}^\ell(d)$ be the optimal constant in 
\begin{equation}\label{torus Hardy}
    \int_{Q_d} |\Delta^{\ell/2} \psi(x)|^2 dx \geq \mathcal{K}_\text{opt}^\ell(d) \int_{Q_d} \frac{|\psi(x)|^2}{\omega(x)^{2\ell}} dx.
\end{equation}
Then we have the following limiting behaviour
$$
\lim_{d \rightarrow \infty} \frac{\mathcal{K}_\text{opt}^\ell(d)}{d^{\ell}} = 2^{-\ell}.
$$
\end{theorem}

Similar to Theorem \ref{thm: discrete asymptotics}, we have explicit bounds on the torus constant, see \eqref{def: inductive constants (torus)} and Theorem \ref{thm: explict torus constant}. As we shall see later in Section \ref{sec: main result proof}, the limiting value of $2^{-\ell}$ corresponds to the function $\psi = e^{ix_1}$. This surprisingly suggests that for large dimensions the optimizers only have contributions coming from small frequencies. 

\subsection{Overview of the arguments}
The main task is to establish Theorem \ref{thm: torus asymptotics}, and the lower bound on the discrete constant $\mathcal{C}_\text{opt}^\ell(d)$ follows from it as follows: integration by parts along with H\"older's inequality gives
$$
\int_{Q_d} |\Delta^{\ell/2} \psi|^2 dx = (-1)^\ell \int_{Q_d} \overline{\psi} \Delta^\ell \psi dx \leq \left(\int_{Q_d} |\Delta^\ell \psi|^2 \omega^{2\ell} dx\right)^{1/2} \left(\int_{Q_d} |\psi|^2 \omega^{-2\ell} dx \right)^{1/2}.
$$
Invoking Hardy inequality on torus \eqref{torus Hardy} along with Lemma \ref{lem: discrete to continuous} proves
$$
\mathcal{C}_\text{opt}^\ell(d) \geq 4^\ell \mathcal{K}_\text{opt}^\ell(d).
$$
Finally applying Theorem \ref{thm: torus asymptotics} one derives
\begin{equation}\label{lower bound on discrete constant}
    \liminf_{d \rightarrow \infty}  \frac{\mathcal{C}_\text{opt}^\ell(d)}{d^\ell} \geq 2^\ell.
\end{equation}
To prove Theorem \ref{thm: discrete asymptotics}, it remains to construct suitable test functions matching the lower bound up to lower order terms in $d$. This will be done in Section \ref{sec: main result proof}.  

While the Hardy inequalities \eqref{torus Hardy} on torus are similar to the classical result \eqref{classical Hardy} and its higher order versions, there are two significant differences. First, instead of the classical inverse square weight $|x|^{-2}$, we are dealing with an equivalent weight $\omega^{-2}$. Since it lacks rotational invariances, its analysis leads to several non-desirable terms involving coupling between different variables. Second, and perhaps the most important one is the presence of ``global" zero average condition compared to ``local" Dirichlet boundary conditions. This makes it hard to extend the existing ideas as they rely heavily on the fact that the functions under consideration vanish on the boundary. 

Our approach uses two well-known ideas from the Dirichlet Hardy-type inequalities: ground state representation formula and Bochner identity. As expected, using them one cannot hope to determine the optimal constant in \eqref{torus Hardy}, as they depend crucially on the presence of the Dirichlet boundary. However, quite remarkably they turn out to be sufficient to capture the exact asymptotic behaviour of the constant. For $\ell = 1$, the result follows from the ground state representation along with the classical Poincar\'e inequality. The major difficulty lies in the case $\ell > 1$. Due to the anisotropic nature of $\omega$, one gets error terms in the weighted version of the first order Hardy inequality (corresponding to $\ell =1$) \eqref{torus Hardy}, making it difficult to prove the result for general $\ell$ by simple iteration. We show that by combining ground state representation formula and Bochner identity with an inductive scheme, one can bypass this difficulty, yielding a clean argument which only requires classical Poincar\'e inequality as an external input.

\begin{remark}[Comparison with independent work of Huang-Ye]\label{rem: comparison with Huang-Ye}
Huang and Ye independently obtained the same asymptotic formula for the optimal constants in the discrete Hardy inequalities, namely
$$
\lim_{d\to\infty}\frac{\mathcal{C}_{\mathrm{opt}}^{\ell}(d)}{d^\ell}=2^\ell.
$$

Although both approaches begin with a Fourier reduction of the lattice inequality to a weighted inequality on the flat torus, the subsequent arguments and intermediate results are substantially different. Our proof establishes torus Hardy inequalities with the singular weights $\omega^{-2\ell}$ by combining the ground state representation formula with a weighted integrated Bochner identity in an iterative scheme. This produces explicit, recursively defined constants for both the torus and lattice inequalities in every dimension $d>2\ell$. The scheme also provides a flexible way to control the error terms arising at successive stages, while avoiding the need for weighted Poincaré inequalities and relying only on the classical \emph{unweighted} Poincaré inequality.

Huang and Ye, by contrast, establish asymptotically sharp inequalities for a family of \emph{positive} powers of $\omega$, comparing weighted energies of different differential orders. These inequalities, however, do not yield Hardy inequalities involving negative powers of $\omega$, and in particular do not imply Theorem \ref{thm: torus asymptotics} and Theorem \ref{thm: explict torus constant} for general $\ell$. Their argument is based on first and second order Hardy identities and on a probabilistic interpretation of the normalized torus weight as the empirical mean of independent random variables. The weighted Poincar\'e inequalities arising in their proof are treated using concentration and entropy methods. Thus, while both approaches lead to the same discrete asymptotic formula, ours yields singular torus Hardy inequalities together with explicit recursive bounds, whereas theirs establishes a family of torus inequalities involving positive powers of the weight $\omega$ and elucidates the underlying high- dimensional concentration phenomenon. The two methods therefore rely on different principal ingredients and produce complementary intermediate results of independent interest.
\end{remark}
The rest of the article is organized as follows. In Section \ref{sec: key estimates}, we prove two key estimates: a ground state representation formula and a weighted integrated Bochner identity, and then specialize them to our setting. In Section \ref{sec: explicit Hardy constants}, we combine these estimates with an iterative scheme to prove the Hardy inequalities \eqref{torus Hardy} and \eqref{def: Hardy constant}, with explicit constants defined inductively in the parameter $\ell$. Finally, in Section \ref{sec: main result proof}, we use these explicit estimates to prove our main results, Theorem \ref{thm: torus asymptotics} and Theorem \ref{thm: discrete asymptotics}.  

\section{Key tools}\label{sec: key estimates}
\subsection{Preliminary identities}
In this section we prove two identities, which will play an important role in our analysis. The first one is the well-known ground state representation formula from spectral theory of differential operators \cite{frank-groundstate, huang-Ye1}. 

\begin{lemma}
Let $V \in C_\text{per}^\infty(\overline{Q_d})$ be real-valued and let $g \in C_\text{per}^\infty(\overline{Q_d})$ be a positive function. Then we have 
\begin{equation}\label{ground state representation}
    \int_{Q_d} V(x) |\nabla \psi|^2 dx = \int_{Q_d} \frac{-\textnormal{div}(V \nabla g)}{g} |\psi|^2 dx + \int_{Q_d} V g^2 \left|\nabla\left(\frac{\psi}{g}\right)\right|^2 dx,  
\end{equation}
for all functions $\psi \in  C_\text{per}^\infty(\overline{Q_d})$.
\end{lemma}

\begin{remark}
Traditionally the identity is stated for functions $\psi$ having compact support and $g$ a positive smooth function in a domain $\Omega$. The proof carries over verbatim to our setting, under an additional periodicity assumption on the function $g$. This additional assumption arises from the absence of Dirichlet boundary conditions for the test functions $\psi$.    
\end{remark}

\begin{proof}
Expanding the integrand in the last term in \eqref{ground state representation}, along with chain rule gives
$$
V g^2 \left|\nabla\left(\frac{\psi}{g}\right)\right|^2 = V |\nabla \psi|^2 + V \frac{|\nabla g|^2}{|g|^2} |\psi|^2 - V \frac{\nabla g}{g} \cdot \nabla |\psi|^2.
$$
Applying integration by parts, along with the periodicity assumptions on $V, g, \psi$ and their derivatives proves the desired identity.
\end{proof}

The second one is a weighted integrated Bochner identity. We introduce the following handy notation:
$$
\psi_i := \partial_{x_i} \psi, \hspace{5pt} \text{and} \hspace{5pt} \psi_{ij} := \partial_{x_j}\partial_{x_i} \psi.
$$

\begin{lemma}
Let $V \in C_\text{per}^\infty(\overline{Q_d})$ be real-valued. Then for $\psi \in  C_\text{per}^\infty(\overline{Q_d})$ the following holds
\begin{equation}\label{integrated Bochner identity}
    \int_{Q_d} V |\Delta \psi|^2 dx = \sum_{i=1}^d \int_{Q_d} V |\nabla \psi_i|^2 + \sum_{1 \leq i, j \leq d} \int_{Q_d} \textnormal{Re}(\psi_i \overline{\psi_j}) V_{ij} dx - \int_{Q_d} \Delta V |\nabla \psi|^2 dx.
\end{equation}
\end{lemma}

\begin{proof}
We begin with the well-known Bochner identity, see for instance \cite[Chapter 9]{bochner} 
\begin{equation}\label{Bochner identity}
    \frac{1}{2} \Delta |\nabla \psi|^2 = \sum_i |\nabla \psi_i|^2 + \textnormal{Re}\left(\nabla \overline{\psi} \cdot \nabla (\Delta \psi)\right).
\end{equation}
Integration by parts and the product rule for derivatives gives
$$
\textnormal{Re} \int_{Q_d} V \nabla \overline{\psi} \cdot \nabla(\Delta \psi) dx = - \int_{Q_d} V |\Delta \psi|^2 dx - \textnormal{Re}\int_{Q_d} (\nabla V \cdot \nabla \overline{\psi}) \Delta \psi dx.  
$$
Using this in the Bochner identity \eqref{Bochner identity} yields
$$
\int_{Q_d} V |\Delta \psi|^2 dx = \sum_i\int_{Q_d} V |\nabla \psi_i|^2 -  \textnormal{Re}\int_{Q_d} (\nabla V \cdot \nabla \overline{\psi}) \Delta \psi dx - \frac{1}{2} \int_{Q_d} \Delta V |\nabla \psi|^2 dx.
$$
An application of integration by parts, together with the product and chain rules, the middle term on the right-hand side can be rewritten as
$$
-\textnormal{Re}\int_{Q_d} (\nabla V \cdot \nabla \overline{\psi}) \Delta \psi dx = \sum_{1 \leq i, j \leq d} \int_{Q_d} \textnormal{Re}(\psi_i \overline{\psi_j}) V_{ij} dx -\frac{1}{2}\int_{Q_d} \Delta V |\nabla \psi|^2 dx.
$$
This completes the proof.
\end{proof}

\begin{remark}
As a side remark, we mention that combining the previous two lemmas yields a Hardy–Rellich identity in the same spirit as the ground state representation formula \eqref{ground state representation}. 
\end{remark}

\subsection{(Almost) weighted Hardy inequalities on torus}
We apply the identities obtained in the previous subsection to derive almost weighted Hardy inequalities of first and second order on the torus. The resulting inequalities contain error terms, which arise from the anisotropic nature of the weight and from the use of tools that are better suited to Dirichlet boundary conditions. 

We define the following constants:
$$
h_k(\beta) := -\beta^2 - \beta(d-2k-2)/2, \hspace{5pt} \text{and} \hspace{5pt} r_k(\beta) = h_k(\beta) + k(d-2k-3)/2,
$$
Henceforth we assume that $k$ is a non-negative integer and $\beta$ is a negative real number.

\begin{lemma}[Weighted Hardy inequality of order one]\label{lem: weight Hardy of order one}
For $d > 2k+2$ the following estimate holds
\begin{equation}\label{weighted Hardy inequality}
    \int_{Q_d} |\nabla \psi|^2 \omega^{-2k} dx -  \beta \int_{Q_d}|\psi|^2 \omega^{-2k} dx \geq h_k(\beta) \int_{Q_d} |\psi|^2 \omega^{-2k-2} dx ,
\end{equation}
for all $\psi \in  C_{\text{per}}^\infty(\overline{Q_d})$.
\end{lemma}

\begin{proof}
Let $\epsilon$ be a non-zero real number and $\omega_\epsilon^2 := \omega^2 + \epsilon^2$. Note that this regularization appears since the weights in \eqref{weighted Hardy inequality} are singular at origin. Taking $V = \omega_\epsilon^{-2k}$ and $g = \omega_\epsilon^{2\beta}$ in the ground state representation \eqref{ground state representation} gives
$$
\int_{Q_d} |\nabla \psi|^2 \omega_{\epsilon}^{-2k} dx \geq \int_{Q_d} \frac{-\text{div} (w_\epsilon^{-2k} \nabla g)}{g}|\psi|^2 dx.
$$
Using basic calculus rules and half angle formulas for $\sin x$ and $\cos x$ we get
\begin{align*}
    \frac{-\text{div}(w_\epsilon^{-2k} \nabla g )}{g} = &- \frac{\beta d}{2} \omega_\epsilon^{-2k-2} - \beta (\beta-k-1) \omega^2 \omega_\epsilon^{-2k-4} \\
    &+ \beta \omega^2 \omega_\epsilon^{-2k-2} + \beta(\beta-k-1)\omega_\epsilon^{-2k} \left(\frac{\sum_j \sin^4(x_j/2)}{\omega_\epsilon^4}\right).    
\end{align*}
Taking limit $\epsilon \rightarrow 0$ and noting that all the singularities are integrable for $d > 2k+2$, invoking dominating convergence theorem proves
\begin{align*}
    \int_{Q_d} |\nabla \psi|^2 \omega^{-2k} dx  \geq &h_k(\beta)  \int_{Q_d} |\psi|^2 \omega^{-2k-2} dx + \beta  \int_{Q_d} |\psi|^2 \omega^{-2k} dx \\
    & + \beta(\beta-k-1)  \int_{Q_d} |\psi|^2 \omega^{-2k}  \left(\frac{\sum_j \sin^4(x_j/2)}{\omega^4}\right) dx.
\end{align*}
Noting that $\beta < 0$ makes the last term non-negative, the result follows. 
\end{proof}

\begin{remark}
In the last line of the proof, we drop the final term by simply bounding it from below by zero. At first sight, this may appear to be a rather wasteful estimate. Indeed, one can obtain a sharper bound for this term by using Hölder's inequality, which gives
$$
\frac{\sum_j \sin^4(x_j/2)}{\omega^4} \geq \frac{1}{d}.
$$
This would lead to improved constants in the final estimates. However, retaining this term in \eqref{weighted Hardy inequality} would make the subsequent analysis significantly more cumbersome. Moreover, it turns out that this contribution does not affect the leading order asymptotics of the optimal constant $\mathcal{K}_{\mathrm{opt}}^\ell(d)$. We therefore discard it in favour of a cleaner proof.
\end{remark}

The next lemma is an application of integrated Bochner identity \eqref{integrated Bochner identity} and would lead to the weighted Hardy inequality of order two.
\begin{lemma}
For $d > 2k + 2$ we have
\begin{equation}\label{splitting Laplacian}
    \begin{split}
        \int_{Q_d} |\Delta \psi|^2 \omega^{-2k} dx \geq \sum_{i} \int_{Q_d} |\nabla \psi_i|^2 \omega^{-2k} dx &+ \frac{k(d-2k-3)}{2} \int_{Q_d} |\nabla \psi|^2 \omega^{-2k-2} dx\\
        &- k \int_{Q_d} |\nabla \psi|^2 \omega^{-2k} dx,    
    \end{split}
\end{equation}
for all $\psi \in  C_{\text{per}}^\infty(\overline{Q_d})$.  
\end{lemma}

\begin{proof}
Let $\epsilon$ be non-zero real number, we regularize as before and set $\omega_\epsilon^2 = \omega^2 + \epsilon^2$. Let $V = \omega_\epsilon^{-2k}$,  then its derivatives are given by 
$$
V_{ij} = -\frac{k}{2} \omega_\epsilon^{-2k-2} \cos x_j \delta_{ij} + \frac{k(k+1)}{4} \sin x_i \sin x_j \omega_\epsilon^{-2k-4}, 
$$
where $\delta_{ij}$ is the Kronecker delta. In particular the Laplacian of $V$ becomes
$$
-\Delta V = \frac{k d}{2} \omega_\epsilon^{-2k-2} -k(k+1) \omega^2 \omega_\epsilon^{-2k-4}- k \omega^2 \omega_\epsilon^{-2k-2} + k(k+1) \omega_\epsilon^{-2k} \left(\frac{\sum_j \sin^4(x_j/2)}{\omega_\epsilon^4}\right). 
$$
Using these formulae in the integrated Bochner identity \eqref{integrated Bochner identity}, taking limit $\epsilon \rightarrow 0$, and using non-negativity of
$$
\frac{k(k+1)}{4} \omega^{-2k-4}\left|\sum_j \sin x_j \psi_j \right|^2 \geq 0, \hspace{5pt} \text{and} \hspace{5pt} k(k+1)\omega^{-2k-4} \sum_j \sin^4(x_j/2)|\nabla \psi|^2 \geq 0,
$$
gives
$$
\sum_{1 \leq i, j \leq d} \int_{Q_d} \textnormal{Re}(\psi_i \overline{\psi_j}) V_{ij} dx \geq - \frac{k}{2} \int_{Q_d} \omega^{-2k-2}|\nabla \psi|^2 dx,
$$
and 
$$
- \int_{Q_d} \Delta V |\nabla \psi|^2 dx \geq \frac{k(d-2k-2)}{2} \int_{Q_d}\omega^{-2k-2} dx - k \int_{Q_d}\omega^{-2k} dx,
$$
for $d > 2k + 2$. Combining these estimate gives the desired result. 
\end{proof}
We combine the obtained estimate with almost weighted Hardy inequality \eqref{weighted Hardy inequality} to obtain weighted Hardy inequality of order two.
\begin{lemma}[Weighted Hardy inequality of order two]\label{lem: weight Hardy of order two}
Let $d > 2k+2$, then the following holds
\begin{equation}\label{weighted second order Hardy inequality}
    \int_{Q_d} |\Delta \psi|^2 \omega^{-2k} dx -  (\beta - k) \int_{Q_d}|\nabla \psi|^2 \omega^{-2k} dx \geq r_k(\beta) \int_{Q_d} |\nabla \psi|^2 \omega^{-2k-2} dx,
\end{equation}
for all $\psi \in  C_{\text{per}}^\infty(\overline{Q_d})$.
\end{lemma}
\begin{proof}
Lemma \ref{lem: weight Hardy of order one} applied to $\psi_i$ gives
$$
\int_{Q_d} |\nabla \psi_i|^2 \omega^{-2k} dx \geq \beta \int_{Q_d}|\psi_i|^2 \omega^{-2k} dx + h_k(\beta) \int_{Q_d} |\psi_i|^2 \omega^{-2k-2} dx ,
$$
Summing with respect to $i$ and using \eqref{splitting Laplacian} proves \eqref{weighted second order Hardy inequality}
\end{proof}

\section{Hardy inequalities with explicit constants}\label{sec: explicit Hardy constants}
\subsection{Hardy inequalities on the torus}
Our goal here is to obtain Hardy inequalities \eqref{torus Hardy} for $\Delta^{\ell/2}$ with explicit constants. Due to the error terms appearing on the right-hand sides of Lemma \ref{lem: weight Hardy of order one} and Lemma \ref{lem: weight Hardy of order two}, these estimates cannot be iterated directly to obtain \eqref{torus Hardy}. It turns out that one can set up an involved iterative scheme, yielding the Hardy inequalities, taking \emph{only} classical Poincar\'e inequality as an external input. In particular, this approach bypasses the use of weighted Poincar\'e inequalities. 

Let $\ell$ be a non-negative integer, we define the constants $\mathcal{K}^\ell(d)$ inductively with respect to $\ell$ as follows:
\begin{equation}\label{def: inductive constants (torus)}
    \begin{split}
        &\mathcal{K}^0(d) = 1, \hspace{9pt} \mathcal{K}^1(d) = \frac{d}{2} \left(1 - \frac{\sqrt{2}}{\sqrt{d}}\right)^2,\\
         \text{for $\ell \geq 2$ even}, \hspace{9pt} &\mathcal{K}^\ell(d) = \sup_{(\alpha, \beta) \in \mathcal{A}_d^\ell }\frac{r_{\ell-2}(\alpha) h_{\ell-1}(\beta)}{\frac{-\alpha + \ell -1 }{\mathcal{K}^{\ell-2}(d)} - \beta \frac{r_{\ell-2}(\alpha)}{ \mathcal{K}^{\ell-1}(d)}},\\
         \text{for $\ell \geq 2$ odd}, \hspace{9pt} & \mathcal{K}^\ell(d) = \mathcal{K}^{\ell-1}(d) \sup_{(\alpha, \beta) \in \mathcal{A}_d^\ell} \frac{h_{\ell-1}(\beta)}{1-\beta},
    \end{split}  
\end{equation}
where the parameter set is given by
$$
\mathcal{A}_d^\ell := \Big\{\alpha, \beta <0 : r_{\ell-2}(\alpha) > 0 \hspace{5pt} \text{and} \hspace{5pt} h_{\ell-1}
(\beta) > 0 \Big\}.
$$
Note that for $d > 2\ell$, the set $\mathcal{A}_d^\ell$ is non-empty and the constants $\mathcal{K}^\ell(d)$ are strictly positive and finite. The finiteness follows from the boundedness of the space $\mathcal{A}_d^\ell$.
\begin{theorem}\label{thm: explict torus constant}
Let $\ell \in \N$ and $\mathcal{K}^\ell(d)$ be the constant defined by \eqref{def: inductive constants (torus)}. Let $d > 2\ell$, then the following Hardy inequalities hold 
\begin{equation}\label{explicit constant (torus)}
    \int_{Q_d} |\Delta^{\ell/2} \psi(x)|^2 dx \geq \mathcal{K}^\ell(d) \int_{Q_d} \frac{|\psi(x)|^2}{\omega(x)^{2\ell}} dx,
\end{equation}
for all functions $\psi \in C_\text{per}^\infty(\overline{Q_d})$ having zero average.
\end{theorem}

\begin{proof}
We prove the result iteratively as is evident from the definition of the constants. We first prove the base case $\ell = 1$. Applying Lemma \ref{lem: weight Hardy of order one} for $k=0$, along with the Poincar\'e inequality
$$
\int_{Q_d} |\psi|^2 dx \leq \int_{Q_d}|\nabla \psi|^2 dx,
$$
gives
$$
(1-\beta)\int_{Q_d} |\nabla \psi|^2 dx \geq h_0(\beta) \int_{Q_d} |\psi|^2 \omega^{-2} dx.
$$
Optimizing with respect to $\beta < 0$ and noting 
$$
\sup_{\beta < 0}\frac{h_0(\beta)}{(1-\beta)} = \frac{d}{2}\left(1- \frac{\sqrt{2}}{\sqrt{d}}\right)^2,
$$
proves the result for $\ell = 1$. The supremum follows by differentiating and checking the sign of the derivative.  

Induction step: Let $m \geq 2$ and assume \eqref{explicit constant (torus)} holds for $1 \leq \ell < m$. We introduce the following useful shorthand notation: 
$$
\mathcal{D}_k^0(\psi) = \int_{Q_d} |\psi|^2 \omega^{-2k} dx, \hspace{9pt} \mathcal{D}_k^1(\psi) = \int_{Q_d} |\nabla \psi|^2 \omega^{-2k} dx, \hspace{9pt}   \mathcal{D}_k^2(\psi) = \int_{Q_d} |\Delta \psi|^2  \omega^{-2k} dx.
$$
In this notation, for $d > 2k+2$, and $\beta$ replaced by $\alpha$, the estimate \eqref{weighted second order Hardy inequality} reads as
$$
\Dt_k^2(\psi) + (-\alpha + k) \Dt_k^1(\psi) \geq r_k(\alpha) \Dt_{k+1}^1(\psi).
$$
Applying Lemma \ref{lem: weight Hardy of order one} to this, gives
\begin{equation}\label{almost weighted Rellich inequality}
    \Dt_k^2(\psi) + (-\alpha + k)\Dt_k^1(\psi) - \beta r_k(\alpha)\Dt_{k+1}^0(\psi) \geq h_{k+1}(\beta)r_k(\alpha) \Dt_{k+2}^0(\psi),
\end{equation}
for $d > 2k+4$, under the assumption that $r_k(\alpha) > 0$.

First we consider the case that $m$ is even. Noting that $\Delta \psi$ has zero average, and using induction hypothesis for  $\ell = m-2$ along with the estimate \eqref{almost weighted Rellich inequality}, we get
\begin{align*}
    \int_{Q_d} |\Delta^{m/2} \psi|^2 dx &\geq \mathcal{K}^{m-2}(d) \Dt_{m-2}^2(\psi) \\
    &\geq \mathcal{K}^{m-2}(d) h_{m-1}(\beta)r_{m-2}(\alpha) \Dt_m^0(\psi)
    \\
    &-(-\alpha + m-2)\mathcal{K}^{m-2}(d) \Dt_{m-2}^1(\psi) + \beta r_{m-2}(\alpha) \mathcal{K}^{m-2}(d) \Dt_{m-1}^0(\psi).
\end{align*}
Applying the induction hypothesis once more, now with $\ell=m-1$ to $\psi$ and with $\ell=m-2$ to $\psi_i := \partial_{x_i} \psi$ (note that $\psi_i$ has zero average as well), we obtain
\begin{align*}
    \int_{Q_d} &|\Delta^{m/2} \psi|^2 dx + (-\alpha + m-2) \int_{Q_d} |\Delta^{(m-1)/2} \psi|^2 dx \\
    & -\frac{\beta r_{m-2}(\alpha)\mathcal{K}^{m-2}(d)}{\mathcal{K}^{m-1}(d)} \int_{Q_d} |\Delta^{(m-1)/2} \psi|^2 dx  \geq \mathcal{K}^{m-2}(d) h_{m-1}(\beta) r_{m-2}(\alpha) \Dt_m^0(\psi).   
\end{align*}
Note that all the constants in the estimate above are positive, since $(\alpha, \beta) \in \mathcal{A}_d^m$. Finally, using classical Poincar\'e inequality (follows from the Fourier expansion of $\psi$)
\begin{equation}\label{Poincare inequality}
    \int_{Q_d} |\Delta^{m/2} \psi|^2 dx \geq \int_{Q_d} |\Delta^{(m-1)/2} \psi|^2 dx,    
\end{equation}
and optimizing with respect to $(\alpha, \beta) \in \mathcal{A}_d^\ell$, the result follows for even $m$.

Next we consider the case that $m$ is odd. Using induction hypothesis for 
$\ell = m-1$ and applying Lemma \ref{lem: weight Hardy of order one} for $k = m-1$ we have
\begin{align*}
    \int_{Q_d} |\Delta^{m/2} \psi|^2 dx &\geq \mathcal{K}^{m-1}(d) \Dt_{m-1}^1(\psi) \\
    & \geq \mathcal{K}^{m-1}(d) h_{m-1}(\beta) \Dt_m^0(\psi) + \beta\mathcal{K}^{m-1}(d) \Dt_{m-1}^0(\psi).
\end{align*}
Another application of induction hypothesis for $\ell = m-1$ gives
$$
\int_{Q_d} |\Delta^{m/2} \psi|^2 dx -\beta \int_{Q_d} |\Delta^{(m-1)/2} \psi|^2 dx \geq \mathcal{K}^{m-1}(d) h_{m-1}(\beta)\Dt_m^0(\psi).
$$
Appealing to Poincar\'e inequality \eqref{Poincare inequality} once more and optimizing with respect to $\beta$ proves the result for odd $m$.
\end{proof}

Computing the exact supremum in \eqref{def: inductive constants (torus)} for general $\ell$ appears to be difficult because of the complexity of the recursion. However, by making suitable choices of the parameters $\alpha$ and $\beta$, guided by numerical experiments, we obtain asymptotically sharp lower and upper bounds for these constants. 

\begin{lemma}\label{lem: limit of torus constant}
Let $\ell \in \N$ and $\mathcal{K}^\ell(d)$ be the constant defined iteratively in \eqref{def: inductive constants (torus)}. Then 
$$
\lim_{d \rightarrow \infty}\frac{\mathcal{K}^\ell(d)}{d^{\ell}} = 2^{-\ell}.
$$
\end{lemma}

\begin{proof}
Let $\ell \geq 2$. For $(\alpha, \beta) \in \mathcal{A}_d^\ell$, the upper bound
$$
h_{\ell-1}(\beta) \leq -\beta d/2,
$$
implies for even $\ell$,
$$
\mathcal{K}^\ell(d) \leq \sup \frac{r_{\ell-2}(\alpha) h_{\ell-1}(\beta) \mathcal{K}^{\ell-1}(d)}{-\beta r_{\ell-2}(\alpha)} \leq d/2 \mathcal{K}^{\ell-1}(d).
$$
For odd $\ell$ we get
$$
\mathcal{K}^\ell(d) \leq \mathcal{K}^{\ell-1}(d) \frac{h_{\ell-1}(\beta)}{-\beta} \leq d/2 \mathcal{K}^{\ell-1}(d). 
$$
Hence, for $\ell \geq 2$ we have
\begin{equation}\label{limsup torus}
    \limsup_{d \rightarrow \infty}\frac{\mathcal{K}^\ell(d)}{d^\ell} \leq \frac{1}{2}\limsup_{d \rightarrow \infty} \frac{\mathcal{K}^{\ell-1}(d)}{d^{\ell-1}}.
\end{equation}
For the lower bound, we make the following choice of parameters: $\alpha = \sqrt{2\ell-2} - \sqrt{d}$ and $\beta = \sqrt{2\ell} - \sqrt{d}$. Note that for $d > 2\ell$, this choice of $(\alpha, \beta) \in \mathcal{A}_d^\ell$, since direct computation gives
$$
h_{\ell-1}(\beta) = \frac{\left(\sqrt{d}- \sqrt{2\ell}\right)^2 \left(\sqrt{d}+\sqrt{2\ell}-2\right)}{2},
$$
and 
$$
r_{\ell-2}(\alpha) = \frac{\left(\sqrt{d}- \sqrt{2\ell-2}\right)^2 \left(\sqrt{d}+\sqrt{2\ell-2} -2\right)}{2} + \frac{\left(\ell-2\right)\left(d-2\ell+1\right)}{2}.
$$
Using these expressions, we deduce 
$$
h_{\ell-1}(\beta) = \frac{d \sqrt{d}}{2} (1 - o(1)), \hspace{9pt} r_{\ell-2}(\beta) \frac{d \sqrt{d}}{2} (1 - o(1)).
$$
For odd $\ell$ this gives
\begin{equation}\label{liminf torus odd l}
    \mathcal{K}^\ell(d) \geq \mathcal{K}^{\ell-1}(d) \frac{h_{\ell-1}(\beta)}{1-\beta} = \frac{d}{2}\mathcal{K}^{\ell-1}(d) (1-o(1)).    
\end{equation}
Similarly for even $\ell$ we get
\begin{equation}\label{liminf torus even l}
    \mathcal{K}^\ell(d) \geq \frac{r_{\ell-2}(\alpha) h_{\ell-1}(\beta)}{\frac{-\alpha + \ell -1 }{\mathcal{K}^{\ell-2}(d)} - \beta \frac{r_{\ell-2}(\alpha)}{ \mathcal{K}^{\ell-1}(d)}}.
\end{equation}
We now prove the result by induction with respect to parameter $\ell$. Assume $\ell \geq 2$ and the result holds true for smaller $\ell's$. If $\ell$ is odd, then the result follows from \eqref{limsup torus} and \eqref{liminf torus odd l}. For even $\ell$, asymptotics of $h$ and $r$ along with induction hypothesis gives
$$
\liminf_{d \rightarrow \infty} \frac{\mathcal{K}^\ell(d)}{d^\ell} \geq 2^{-\ell}.
$$
To obtain this we used
$$
\frac{-\alpha+\ell-1}{\mathcal{K}^{\ell-2}(d)} = \mathcal{O}(d^{5/2-\ell}),
$$
and
$$
\lim_{d \rightarrow \infty} -\beta \frac{r_{\ell-2}(\alpha)}{\mathcal{K}^{\ell-1}(d) d^{3-\ell}} = 2^{\ell-2}.
$$

Using \eqref{limsup torus} once more proves the result for even $\ell$. It remains to check the result for $\ell=1$, which follows from the exact expression of the constant $\mathcal{K}^1(d)$ from \eqref{def: inductive constants (torus)}.
\end{proof}

\subsection{Hardy inequalities on the lattice}
Based on the Fourier reduction technique, we derive discrete Hardy inequalities using the results from the previous section. We define the corresponding constants iteratively: Let $\mathcal{A}^\ell_d$ be the parameter set defined in \eqref{def: inductive constants (torus)}, then we define
\begin{equation}\label{def: inductive constants (lattice)}
    \begin{split}
        &\mathcal{C}^0(d) = 1, \hspace{9pt} \mathcal{C}^1(d) = 2d \left(1 - \frac{\sqrt{2}}{\sqrt{d}}\right)^2,\\
         \text{for $\ell \geq 2$ even}, \hspace{9pt} &\mathcal{C}^\ell(d) = \sup_{(\alpha, \beta) \in \mathcal{A}_d^\ell }\frac{4 r_{\ell-2}(\alpha) h_{\ell-1}(\beta)}{\frac{-\alpha + \ell -1 }{4\mathcal{C}^{\ell-2}(d)} - \beta \frac{r_{\ell-2}(\alpha)}{ \mathcal{C}^{\ell-1}(d)}},\\
         \text{for $\ell \geq 2$ odd}, \hspace{9pt} & \mathcal{C}^\ell(d) = 4\mathcal{C}^{\ell-1}(d) \sup_{(\alpha, \beta) \in \mathcal{A}_d^\ell} \frac{h_{\ell-1}(\beta)}{1-\beta},
    \end{split}  
\end{equation}
We note that $\mathcal{C}^\ell(d) = 4^\ell \mathcal{K}^\ell(d)$. This follows from the definition of these constants.
\begin{theorem}\label{thm: explict lattice constant}
Let $\ell \in \N$ and $\mathcal{C}^\ell(d)$ be the constant as defined in \eqref{def: inductive constants (lattice)}. Then for $d > 2\ell$, the following discrete Hardy inequalities hold 

\begin{equation}\label{explicit lattice constant}
    \sum_{n \in \Z^d} |\Delta^{\ell/2} u(n)|^2 \geq \mathcal{C}^\ell(d) \sum_{n \in \Z^d} \frac{|u(n)|^2}{|n|^{2\ell}},
\end{equation}
for all functions $u \in C_c(\Z^d)$ with $u(0) = 0$. Moreover, 
\begin{equation}\label{lattice constant limit}
    \lim_{d\rightarrow \infty}\frac{\mathcal{C}^\ell(d)}{d^\ell} = 2^\ell.
\end{equation}
\end{theorem}

\begin{proof}
Appealing to the Fourier reduction in Lemma \ref{lem: discrete to continuous}, it is sufficient to prove the corresponding inequality on the torus. To this end, applying integration by parts and H\"older's inequality gives
$$
\int_{Q_d} |\Delta^{\ell/2} \psi|^2 dx = (-1)^\ell \int_{Q_d} \overline{\psi} \Delta^\ell \psi dx \leq \left(\int_{Q_d} |\Delta^\ell \psi|^2 \omega^{2\ell} dx\right)^{1/2} \left(\int_{Q_d} |\psi|^2 \omega^{-2\ell} dx \right)^{1/2}.
$$
Next applying Theorem \ref{thm: explict torus constant}, and lifting the inequality back to the lattice via Lemma \ref{lem: discrete to continuous} and finally noting that $\mathcal{C}^\ell(d) = 4^\ell \mathcal{K}^\ell(d)$ proves the result. 

The limit of $C^\ell(d)$ follows from $\mathcal{C}^\ell(d) = 4^\ell \mathcal{K}^\ell(d)$ and Lemma \ref{lem: limit of torus constant}.
\end{proof}

\section{Asymptotically sharp lower and upper bounds}\label{sec: main result proof}
\subsection{Proof of Theorem \ref{thm: torus asymptotics}}
Theorem \ref{thm: explict torus constant} along with the Lemma \ref{lem: limit of torus constant} proves
$$
\liminf_{d \rightarrow \infty} \frac{\mathcal{K}_\text{opt}^\ell(d)}{d^\ell} \geq 2^{-\ell}.
$$
It remains to show that this lower bound is attained. To do so, take $\psi = e^{i x_1}$, which is smooth and has zero average. Using this as our test function along with Jensen's inequality yields
$$
\mathcal{K}_\text{opt}^\ell(d) \leq |Q_d| \left(\int_{Q_d} \omega^{-2\ell} dx\right)^{-1} \leq |Q_d|^{-\ell} \left(\int_{Q_d} \omega^2 dx\right)^\ell = (d/2)^\ell.
$$
This proves 
\begin{align*}
    \limsup_{d \rightarrow \infty} \frac{\mathcal{K}_\text{opt}^\ell(d)}{d^\ell} \leq 2^{-\ell}. \tag*{\qedsymbol}
\end{align*}

\subsection{Proof of Theorem \ref{thm: discrete asymptotics}}
Theorem \ref{thm: explict lattice constant} gives the lower bound
$$
\liminf_{d \rightarrow \infty} \frac{\mathcal{C}_\text{opt}^\ell(d)}{d^\ell} \geq 2^\ell. 
$$
For upper bound we take the $u = \chi_{\{|n|=1\}}$, where $\chi_\Omega$ is the characteristic function of $\Omega$.  
\begin{equation}\label{lattice constant upper bound}
    \mathcal{C}_\text{opt}^\ell(d) \leq \frac{1}{2d} \sum_{n \in \Z^d} |\Delta^{\ell/2} u|^2 = \frac{4^\ell}{2d} \int_{Q_d} |\widehat{u}|^2 \omega^{2\ell} dx,
\end{equation}
where $\widehat{u}(x) := (2\pi)^{-\frac{d}{2}}\sum\limits_{n \in \Z^d}u(n) e^{-i n \cdot x}$ is the Fourier transform of $u$. The last identity follows from Parseval's identity, see \cite[Lemmas 2.2, 2.3]{gupta} for details. Using half angle formula for $\cos$ and binomial theorem, the integral can be expressed as
\begin{align*}
    \int_{Q_d}|\widehat{u}|^2 \omega^{2\ell} dx &= 4(2\pi)^{-d} \frac{d^\ell}{2^\ell} \sum_{k=0}^\ell (-1)^k {\ell \choose k} \int_{Q_d}\left(\sum_{j=1}^d \cos x_j\right)^2 \left(\frac{1}{d}\sum_{j=1}^d \cos x_j \right)^{k} dx\\
    & = 2d \frac{d^\ell}{2^\ell} + 4(2\pi)^{-d} \frac{d^\ell}{2^\ell} \sum_{k=2}^\ell (-1)^k {\ell \choose k} \int_{Q_d}\left(\sum_{j=1}^d \cos x_j\right)^2  \left(\frac{1}{d}\sum_{j=1}^d \cos x_j \right)^{k} dx.
\end{align*}
In the last step we used
$$
\int_{Q_d} \left(\sum_j \cos x_j \right)^3 dx = 0.
$$
Next we estimate the last term, which we denote by $\mathcal{R}^\ell_d$. Noting that $1/d \sum_j \cos x_j \leq 1$ in the absolute value we get
\begin{align*}
    \mathcal{R}^\ell_d &\leq 4 (2\pi)^{-d} \frac{d^\ell}{2^\ell} \left(\sum_{k=2}^\ell {\ell \choose k}\right) \int_{Q_d} \left(\sum_{j=1}^d \cos x_j\right)^2  \left(\frac{1}{d}\sum_{j=1}^d \cos x_j \right)^{2} dx\\
    & \leq 4(2\pi)^{-d} d^{\ell-2} \int_{Q_d}  \left(\sum_{j=1}^d \cos x_j \right)^{4} dx = d^{\ell-2}\left(3 d^2 - \frac{3d}{2}\right). 
\end{align*}
Using this estimate back in \eqref{lattice constant upper bound} gives 
\begin{equation}\label{limsup of lattice constant}
    \mathcal{C}_\text{opt}^\ell(d) \leq 2^\ell d^\ell + \frac{3}{2} \cdot 4^\ell d^{\ell-1} - 3 \cdot 4^{\ell-1} d^{\ell-2}.  
\end{equation}
This proves the desired upper bound
\begin{align*}
    \limsup_{d \rightarrow \infty}  \frac{\mathcal{C}_\text{opt}^\ell(d)}{d^\ell} \leq 2^\ell. \tag*{\qedsymbol}
\end{align*}

\subsection*{Acknowledgments}
We thank Yehuda Pinchover for his careful reading and for his valuable suggestions.


\end{document}